\documentclass[12pt]{amsart}

\usepackage{amsmath}
\usepackage{amssymb}
\usepackage{a4wide}
\usepackage{xcolor}

\newtheorem{theorem}{Theorem}[section]
\newtheorem{lemma}[theorem]{Lemma}
\newtheorem{note}[theorem]{Note}

\newtheorem{cor}[theorem]{Corollary}

\newtheorem*{Theorem1'}{Theorem 1'}

\theoremstyle{definition}

\newtheorem{example}[theorem]{Example}

\theoremstyle{remark}

\numberwithin{equation}{section}

\newcommand \BB{{\mathcal C}}
\newcommand \C{{\mathbb C}}
\renewcommand \O{{\mathcal O}}
\renewcommand \ker{{\mathrm {ker}}}

\newcommand \GL{{\mathrm {GL}}}
\newcommand \Sp{{\mathrm {Sp}}}
\newcommand \U{{\mathrm {U}}}

\newcommand \SL{{\mathrm {SL}}}
\newcommand \SSL{{\mathrm {SSL}}}
\newcommand \m{{\mathfrak {m}}}

\renewcommand \r{{\mathfrak {r}}}

\newcommand \va{{\varepsilon}}
\newcommand \ind{{\mathrm {ind}}}

\newcommand \F{{\mathbb F}}

\newcommand \rk{{\mathrm {rank}}}

\newcommand{\ssl}{\SSL^+_*(2,A)}
\newcommand{\sll}{\SL^+_*(2,A)}

\begin{document}

\title[Weil representations]{Weil representations via abstract data and
Heisenberg groups: a comparison}

\author{J. Cruickshank}
\address{School of Mathematics, Statistics and Applied Mathematics, National University of Ireland, Galway, Ireland}
\email{james.cruickshank@nuigalway.ie}

\author{L. Guti\'errez Frez}
\address{Instituto de Ciencias Fisicas y Matem\'aticas, Universidad Austral de Chile, Chile}
\email{luis.gutierrez@uach.cl}

\author{F. Szechtman}
\address{Department of Mathematics and Statistics, Univeristy of Regina, Canada}
\email{fernando.szechtman@gmail.com}
%\thanks{The second author was supported in part by an NSERC discovery grant}


%\date{January 1, 2001 and, in revised form, June 22, 2001.}



\begin{abstract} Let $B$ be a ring, not necessarily commutative, having an involution $*$ and let $\U_{2m}(B)$
be the unitary group of rank $2m$ associated to a hermitian or skew hermitian form relative to $*$. When $B$ is finite, we construct a Weil representation of
$\U_{2m}(B)$ via Heisenberg groups and find its explicit matrix form on the Bruhat elements. As a consequence, we derive information
on generalized Gauss sums. On the other hand, there is an axiomatic
method to define a Weil representation of $\U_{2m}(B)$, and we compare the two Weil representations thus obtained under fairly general
hypotheses. When $B$ is local, not necessarily finite, we compute the index of the subgroup of $\U_{2m}(B)$ generated by its Bruhat elements.
Besides the independent interest, this subgroup and index are involved in the foregoing comparison of Weil representations.
\end{abstract}

\maketitle

\section{Introduction}

Weil representations were introduced by A. Weil \cite{W} for symplectic groups over local fields. Following Weil's ideas, analogues over finite fields
were constructed in \cite{H} and \cite{Ge}, although these representations had already been considered, independently of Weil's work, in \cite{BRW} and \cite{Wa}.

Over the years, Weil representations of symplectic and unitary groups have attract significant attention, see for instance \cite{AP, CMS, G, GPS, S, T, TZ, GMT}.

In this paper, we construct, in explicit matrix form, Weil representations of unitary groups $\U(2m,B)$, $B$ a finite ring,
by means of two methods which are subsequently compared; as a consequence we derive identities for generalized Gauss sums, analogous to those encountered in the classical case when $B$ is a finite field.
Our comparison involves the subgroup of $\U(2m,B)$ generated by its Bruhat elements and, for this reason, we determine the exact relationship between this subgroup and 
$\U(2m,B)$ when $B$ is local, not necessarily finite.

The connection between the Weil representation of the symplectic group $\Sp(2m,F_q)$ and Gauss sums over $F_q$, $q$ odd,
is well-known, as Gauss sums are required to correct an initially projective Weil representation into an ordinary one
and, once this is achieved, various Weil character values turn out to be Gauss sums. Let us thus
start by describing the generalized Gauss sums arising from the Weil representation of $\U(2m,B)$ and the information on the former that can be gleaned from the latter.

Given an odd prime prime $p$ and an integer $t$ not divisible by $p$, we have the Gauss sum
$$
G_t=\underset{b\in F_p}\sum \beta(t b^2),
$$
where $\beta:F_p^+\to\C^\times$ is the group homomorphism given by $\beta(1)=e^{2\pi i /p}$. Gauss proved~that
\begin{equation}
\label{gaus1}
 G_1=\begin{cases} \sqrt{p} & \text{ if }p\equiv 1\mod 4,\\ i\sqrt{p} &
 \text{ if }p\equiv 3\mod 4.\end{cases}
\end{equation}
Landau's book \cite[pp. 197-218]{La} contains four different proofs of this result.

It is well-known \cite[ch. 4]{Ri} that $G_t$ and $G_1$ are connected by
\begin{equation}
\label{gaus2}
G_t=\left(\frac{t}{p}\right)G_1=(-1)^{\nu(t)}G_1,
\end{equation}
where $\left(\frac{t}{p}\right)$ is the Legendre symbol and $\nu(t)$ is the number of integers $1\leq j\leq (p-1)/2$ such that $tj\equiv -k$ for some $1\leq k\leq (p-1)/2$. From (\ref{gaus1}) and (\ref{gaus2})  we obtain
\begin{equation}
\label{gaus3}
G_t^2=(-1)^{(p-1)/2} p,
\end{equation}
which can also be derived directly (see  \cite[ch. 4]{Ri}).

Let $B$ be a finite ring, not necessarily commutative, such that $2\in B^\times$, the unit group of~$B$. We also assume that $B$ has an involution $*$, that is,
an antiautomorphism of order 1 or 2. We suppose, as well,
that $B$ admits a linear character $\beta:B^+\to\C^\times$ whose kernel contains no right ideals except~(0). For a finite ring $R$ this condition has
been extensively studied, it is left-right symmetric and equivalent to $R$ being a Frobenius ring; see \cite{Ho} and references therein.

For a correct generalization of Gauss sum, we require that $\beta(b+\va b^*)=1$ for all $b\in B$, where $\va$ is a fixed element taken from $\{1,-1\}$.

Extend $*$ to an involution, also denoted by $*$, of the full matrix ring $M(m,B)$ by declaring $(Y^*)_{ij}=(Y_{ji})^*$, for $Y\in M(m,B)$. Suppose $T\in\GL(m,B)$ satisfies $T+\va T^*=0$. If $\va=1$ we further assume that $m=2n$ is even in order to ensure the existence of such~$T$.

Associated to this set-up we have the generalized Gauss sum
\begin{equation}
\label{gaus5}
G_T=\underset{b\in B^m}\sum \beta(b^* T b),
\end{equation}
where $B^m$ is viewed here as the space of column vectors. As in (\ref{gaus2}), we wish to connect (\ref{gaus5}) to a fixed standard sum $G_S$, where
$$
S=\begin{cases} 1_m  & \text{ if }\va=-1,\\
\begin{pmatrix} 0 & 1_n \\  -1_n & 0 \end{pmatrix}& \text{ if }\va=1.
\end{cases}
$$
In order to realize the desired similitude with the classical case we need an analogue of~$\nu(t)$. For this purpose, let $I$ be any subset of $B^m$
obtained by selecting exactly one element from each of the sets $\{v,-v\}$, $v\in B^m$, $v\neq 0$. For $Y\in\GL(m,B)$, let
$I_Y=\{v\in I\,|\, Yv\in -I\}$. Then the function $\mu:\GL(m,B)\to \{1,-1\}\subset \C^\times$, given by $Y\to (-1)^{|I_Y|}$
is a group homomorphism independent of the choice of $I$ (see \S\ref{formula} below).

With this notation, we have the following
results, which are solely based on our work on the Weil representation of $\U(2m,B)$. Theorems \ref{mu1} and \ref{mu2} give
$$
G_T=\mu(T) G_S,
$$
$$
G_T^2=\va^{(|B|^m-1)/2}|B|^m,
$$
while Theorem \ref{mu55} shows that
\begin{equation}
\label{gaus6}
\mu(Y)=\mu(Y^*),\quad Y\in\GL(m,B).
\end{equation}
By considering $I^*$ instead of $I$ we deduce from (\ref{gaus6}) that $\mu$ remains invariant if we consider instead $B^m$ as a row space
and let $\GL(m,B)$ act on $B^m$ from the right.

Let us now describe our main results regarding the Weil representation. By definition, $\U=\U(2m,B)$ is the subgroup of all $Y\in\GL(2m,B)$ such that
$$
Y^* \begin{pmatrix} 0 & 1_m \\  \va 1_m & 0 \end{pmatrix} Y=\begin{pmatrix} 0 & 1_m \\  \va 1_m & 0 \end{pmatrix}.
$$

As in the classical case when $B$ is a finite field and $*$ is the identity, there is a Heisenberg group $H$ and a Schr$\mathrm{\ddot{o}}$dinger representation $S:H\to\GL(X)$ of type $\beta$
and degree $|B|^m$ that remains invariant under the natural action of $\U$ on $H$. This gives rise to a Weil representation $W:\U\to\GL(X)$
satisfying
$$
W(g)S(h)W(g)^{-1}=S({}^g h),\quad g\in \U, h\in H.
$$
Each $W(g)$ is an intertwining operator between $S$ and its conjugate representation $S^g$, and the map $g\mapsto W(g)$ is a group homomorphism. This material is expounded in \S\ref{schro} and \S\ref{formula}.

Our first goal is to obtain an explicit matrix description of $W(g)$ for each Bruhat element $g\in\U$. This is achieved in Theorem \ref{weilrepskewh}
in the skew hermitian case $\va=-1$ and in Theorem \ref{weilrepher} in the hermitian case $\va=1$. We first determine a projective representation
$P:\U\to\GL(X)$ consisting of intertwining operators $P(g)$ and then find correcting scalars $c(g)\in\C^\times$ such that $W(g)=P(g)c(g)$, with $W$ being
an ordinary representation. The work to compute the projective representation $P$ and the correcting factor $c$ is begun
in \S\ref{formula} and fully developed in \S\ref{ela} and \S\ref{fela}. At the very end of these calculations we split
our work into the skew hermitian case, done in \S\ref{skewsec}, and the hermitian case, done in \S\ref{sech}. These last two sections
also include the derivation of our prior formulas on generalized Gauss sums. Many of our difficulties are related to
the generality of the ring $B$.

Now associated to a ring $A$ with involution, also denoted by $*$, \cite{GPS} considers the group $\SL_*^\va(2,A)$,
which coincides with $\U(2m,B)$ when $A=M(m,B)$. Details can be found in \S\ref{uno}. Moreover, \cite{GPS} develops
an axiomatic method to produce a generalized Weil representation of $\SL_*^\va(2,A)$, provided $\SL_*^\va(2,A)$ is generated
by its Bruhat elements subject to certain canonical defining relations, namely the relations (R1)-(R6) found in~\S\ref{uno}.  
In \S\ref{comop} we suitably modify the axioms developed \cite{GPS} to produce a function defined on the Bruhat elements of $\U(2m,B)$ that preserves the relations (R1)-(R6) and coincides with the formulas
from Theorems \ref{weilrepskewh} and \ref{weilrepher} for a carefully chosen set of initial parameters required by the axioms. In simple terms, the Heisenberg and axiomatic
methods produce the same Weil representation. This is our second goal, achieved in Theorem \ref{main}, our main result.

The above theory is applicable to several families of rings, as shown in \S\ref{exam}. Needless to say, the classical cases of
Weil representations of $\Sp(2m,F_q)$ and $\U(2m,F_{q^2})$ are but very special members of these families.

Set $A=M(m,B)$. Then the Bruhat elements of $\U(2m,B)=\SL_*^\va(2,A)$ are, by definition:
$$
\omega=\begin{pmatrix} 0 & 1_m \\ \va 1_m & 0 \end{pmatrix},\; h_T=\begin{pmatrix} T & 0 \\ 0 & (T^*)^{-1} \end{pmatrix}, T \in \GL(m,B),\; u_S=\begin{pmatrix} 1_m & S \\ 0 & 1_m \end{pmatrix}, S+\va S^*=0.
$$
Let $\SSL_*^\va(2,A)$ be the subgroup of $\SL_*^\va(2,A)$ generated by its Bruhat elements. Our comparison of Weil representations involves $\SSL_*^\va(2,A)$
and we direct attention in \S\ref{sec:special} and \S\ref{sec:general} to study the relationship between $\SSL_*^\va(2,A)$ and $\SL_*^\va(2,A)$ when $B$ is local
but not necessarily finite. It turns out that $\SL_*^\va(2,A)=\SSL_*^\va(2,A)$ except when $\va=1$, $m$ is even and $*$ is ramified (as defined in \S\ref{uno}),
in which case $[\SL_*^\va(2,A):\SSL_*^\va(2,A)]=2$. Precise details are given in \S\ref{sec:special} and \S\ref{sec:general}. There is some subtlety to these results.
Indeed, the case $\va=-1$ follows smoothly from the case when $B=D$ is a division ring, first considered in \cite{PS}. However, the case $\va=1$ requires substantial effort
and it seems to be incorrectly stated in \cite{PS2}, as Proposition 8 therein would imply that $\SL_*^\va(2,A)=\SSL_*^\va(2,A)$, which is certainly not always true, as mentioned above. 
Whether $\va=1$ or $\va=-1$, we provide a list of generators for $\SL_*^\va(2,A)$ when $B$ is local. This is applicable to the classical cases when $\SL_*^\va(2,A)$ becomes
the symplectic group $\Sp(2m,F_q)$, the unitary group $\U(2m,F_{q^2})$ or the orthogonal group $\mathrm{O}(2m,F_q)$, as well as the more general case when $B$ is local and 
these classical groups are factors of $\SL_*^\va(2,A)$.

\section{The groups $\U(2m,B)$, $\SL_*^\va(2,A)$ and  $\SSL_*^\va(2,A)$ }\label{uno}

All rings in this paper are assumed to have an identity element different from zero. The unit group of a ring $B$ will be denoted by $B^\times$.

Let \( B \) be a ring, not necessarily commutative, having an involution $*$, that is, an antiautomorphism of order 1 or 2. We may extend $*$
to an involution, also denoted by $*$, of the full matrix ring $A=M(m,B)$ by declaring $(x^*)_{ij} = (x_{ji})^*$ for every $x\in A$.

We take $\varepsilon$ in $\{-1,1\}$ and let $V$ be a right $B$-module endowed with an $\va$-hermitian form $h:V\times V\to B$. This means that $h$ is $B$-linear in
the second variable and satisfies
$$
h(u,v)^*=\va h(v,u),\quad u,v\in V.
$$
Thus $h$ is hermitian $\varepsilon=1$ and $h$ is skew hermitian $\varepsilon=-1$.
We further assume that $V$ has a basis $\BB=\{u_1,\dots,u_m,v_1,\dots,v_m\}$ relative to which the Gram matrix of $h$ is equal to
$$ J = \begin{pmatrix}0 & 1_m \\ \va 1_m & 0
\end{pmatrix}.
$$
The unitary group $\U=\U(2m,B)$ is the group of all $g\in\GL(V)$ preserving $h$, in the sense
$$
h(gu,gv)=h(u,v),\quad u,v\in V.
$$
Let $g\in\GL(V)$ and suppose that the matrix $M_{\BB}(g)$ of $g$ relative to $\BB$ is equal to
$$
x=\begin{pmatrix} a & b \\ c & d \end{pmatrix} \in \GL(2,A).
$$
Then $g\in\U$ if and only if $x^* J x=J$, which translates as follows:
\begin{equation}
\label{eqelements}
a^*c =-\va c^*a,\; b^*d =-\va d^*b,\; d^*a +\va b^*c = 1.
\end{equation}
Following \cite{PS}, the group of all $x\in\GL(2,A)$ satisfying (\ref{eqelements}) will be denoted by
$\SL^\va_*(2,A)$. Thus the map $\U\to \SL^\va_*(2,A)$, given by $g\mapsto M_{\BB}(g)$, is a group isomorphism.

Set  $ A^{\varepsilon\text{-sym}}=\{s\in A\,|\, s+\varepsilon s^*=0\}$. Following \cite{PS} we define the Bruhat elements of
\( \SL^\va_*(2,A) \) by
$$
    \omega = \begin{pmatrix} 0&1 \\ \va & 0 \end{pmatrix},\;
    h_t = \begin{pmatrix} t & 0 \\ 0 & (t^*)^{-1} \end{pmatrix}, t \in A^{\times},\;
    u_r = \begin{pmatrix} 1 & r \\ 0 & 1 \end{pmatrix}, r \in A^{\varepsilon\text{-sym}}.
$$
One easily verifies that the following relations hold in \( \SL^\va_*(2,A)\):

(R1) $h_sh_{t}=h_{st},\; s,t\in A^\times$;

(R2) $u_qu_{r}=u_{q+r},\;  q,r \in A^{\varepsilon\text{-sym}}$;

(R3) $\omega^2=h_{\va}$;

(R4) $h_tu_r=u_{trt^*}h_t,\; t\in A^\times, r \in A^{\varepsilon\text{-sym}}$;

(R5) $\omega h_t=h_{t^{*-1}}\omega,\; t\in A^\times$;

(R6) $u_t \omega u_{-\va t^{-1}}\omega u_t = \omega h_{-t^{-1}},\; t\in A^\times\cap A^{\varepsilon\text{-sym}}$.

Let $\SSL^\va_*(2,A)$ be the subgroup of $\SL^\va_*(2,A)$ generated by the Bruhat elements.

We wish to determine the index of $\SSL^\va_*(2,A)$ in $\SL^\va_*(2,A)$. For this purpose, we suppose until the end of \S\ref{sec:general} that $B$ is a local ring with Jacobson radical $J(B)=\r$ and corresponding division ring $D=R/\r$. We also assume that $2\in B^\times$. We say that $*$ is ramified if it induces the identity map on $D$, in which case $D=\F$ is a field, and unramified otherwise. 
We set $\overline{A}=M_m(D)$.

\section{On the index of \( \SSL^{+}_*(2,A) \) in \( \SL^{+}_*(2,A)\): a special case}\label{sec:special}

We keep the assumptions and notation from \S\ref{uno}. We suppose in this section that $\va=1$ and $*$ is ramified.

If \( X = \begin{pmatrix} a&b\\c&d \end{pmatrix} \in M(2,A)\),
    we say that the \( (1,1) \)-block of \( X \) is \( a \), the \( (1,2) \)-block of \( X \) is \( b \) and so on.
    Let \( e \in A\) be the matrix whose only nonzero entry is a \( 1 \) in
    position \( (m,m) \), \( k = 1_m - e \) and
    \( T = \begin{pmatrix} k  & e \\ e & k \end{pmatrix} \in M(2,A)\). Then \( T^2 = 1_{2m} \) and,
    by (\ref{eqelements}), \( T\in \sll \).

We assume for the remainder of this section that $\r=0$. Thus for \( x \in A \), \( x^* \) is the transpose of \( x \).
Also, \( \det(T) = -1 \), \( \det(u_s) = \det(h_t) = 1 \) and 
		\( \det(w)= (-1)^m \). Hence for even \( m \), \( \det(Y) = 1 \) for all
    \( Y \in \ssl \). In particular \( T \not\in \ssl \) if \( m \) is even.

Let $Z=\F^m$ be the space of column vectors. The following two results are adaptations of \cite[Propositions 3.2 and 3.3]{PS} to
deal with the case $\va=1$.

\begin{lemma}
    Let \( W \) be a linear subspace of \( Z\), with
    \( \dim W \) odd (resp. even), and let $u\in W$ be nonzero (resp. zero). Then there exists \( s \in A \) such that
    \begin{itemize}
        \item \( s^* = -s \);
        \item \( \ker(s) \cap W = \F u \);
        \item \( sW = sZ \);
        \item \( sW \cap W^\perp = 0\) where \( W^\perp = \{ v \in Z|\, v^*w = 0\text{ for all }w \in W\} \).
    \end{itemize}
    \label{trans1}
\end{lemma}

\begin{proof}
    Since \( \dim W \) is odd (resp. even) we can choose a skew symmetric form
    \( C \) on \( W \) whose kernel is \( \F u \).
    In other words \( C(w,W) = 0 \) if and only if \( w \in \F u \).
    Suppose that \( Z = W \oplus T \) and extend \( C \) to
    \( Z \) by \( C(w_1+t_1,w_2+t_2) = C(w_1,w_2) \). Now
    there is some \( s \in A \) such that \( C(v_1,v_2)
    = v_1^*sv_2\). One readily checks that \( s \) has all the required
    properties.
\end{proof}

\begin{lemma}
    Suppose that \( a, c \in A \) satisfy \( a^*c = -c^*a \), \( a \)
    has odd (resp. even) nullity and \( Aa + Ac = A \). Let \( u \in \ker(a) \)
    be nonzero (resp. zero). Then there is some \( s \in A \)
    such that \( s^* = -s  \) and \( \ker(a+sc) =  \F u  \).
    \label{coprime2}
\end{lemma}

\begin{proof}
    Clearly \( Aa+Ac = A\) implies that \( \ker(a) \cap \ker(c) = 0 \). Thus
    \( c: \ker(a) \cong c\ker(a) \). Also
    \( (ax)^*(cy) = -(cx)^*(ay) \), so \( c\ker(a) \subseteq (aZ)^\perp \). But
    \( \dim( (aZ)^\perp) = \dim(\ker(a)) = \dim(c\ker(a)) \) so \( c\ker(a) =
    (aZ)^\perp\). Equivalently \( (c\ker(a))^\perp = aZ \).

    Now apply Lemma \ref{trans1}
    to the space \( W=c\ker(a) \) and the vector $cu\in W$.  We find a
    skew symmetric \( s \in A \) such that \( \ker(s) \cap W = \F cu \),
    \( s W = sZ \), and \( s W \cap aZ = sW\cap W^\perp=0 \).
    Suppose that \( (a+sc)v = 0 \) for some \( v \in V \). Then
    \( av = -scv \). But \( -scv \in sZ = sW \). So \( av \in
    aZ \cap s W = 0\). Thus \( v \in \ker(a) \). Also \( scv = 0  \).
    Now since \( \ker(s) \cap W = \F cu \)  we see that \( cv
    \in \F cu\) and so \( v \in \F u \). Thus \( \ker(a+sc) = \F u \).
\end{proof}

%\begin{proof}
%    Again we follow \cite{PS}, but now using Lemma \ref{trans2}.
%    Now we find that given a nonzero \( w \in \ker a \) we can
%    find skew symmetric \(s\) so that \( \ker(a+sc) = Bw \).
%\end{proof}

    \begin{lemma}
        Suppose that \( X \in \sll \) and that at least one of \( a,c \)
        has even nullity. Then \( X \in \ssl \).
        \label{lem:evencorankcondition}
    \end{lemma}

    \begin{proof}
        Without loss of generality we can assume that \( a \) has even nullity (replace
        \( X \) by \( wX \) if necessary). Now by Lemma \ref{coprime2}
        there is some \( s = -s^*\) such that \( a+sc \in A^\times \).
        Now the \( (1,1) \)-block of \( u_sX \) is an element of
        \( A^\times \) and the result follows easily.
    \end{proof}

    \begin{theorem}
        Suppose that \( m \) is odd. Then \( \ssl = \sll \).
        \label{thm:indexoddcase}
    \end{theorem}

    \begin{proof}
        Suppose \( X \in \sll \). Clearly for any \( Z \in \ssl \), we have
        $$ X \in \ssl \Leftrightarrow ZX \in \ssl \Leftrightarrow XZ \in
        \ssl.$$ We use this observation repeateadly to reduce \( X \)
        to a suitably nice form.
        By Lemmas \ref{coprime2} and \ref{lem:evencorankcondition}
        we can assume that \( c \) has nullity 1.
        Now there are
        \( t_1,t_2 \in \GL(m,\F ) \) such that \( t_1ct_2 = k\).
        By replacing \( X \) with \( h_{t_1^{*-1}}Xh_{t_2} \)
        we can assume that
        \( c = k \).
        By (\ref{eqelements}) we see that \( a^*k = -ka \). So
        \( (kak)^* = ka^*k = -k^2a = -ka\). Also $-kak = a^*k^2 =
        a^*k = -ka$. In other words \( (kak)^* = -kak \).
        Now, replacing \( X \) by \( u_{-kak}X \) we can arrange that
        the first \( m-1 \) rows of \( a \) are \( 0 \).
        %\( a = \begin{pmatrix} 0 & 0 \\ * & a_{mm} \end{pmatrix}\).
        But \( Aa+Ac = A \) so \( a \neq 0 \).
        Thus \( a \) has nullity \( m-1 \) which is even. It
        follows from Lemma \ref{lem:evencorankcondition}
        that \( X \in \ssl \).
    \end{proof}

We want to use a similar idea to show that $[\SL_*^+(2,A):\SSL_*^+(2,A)]=2$ when $m$ is even. However the argument
    is complicated by the fact that we do not know, a priori, that 
    \( \ssl \unlhd \sll \). So first we need to prove the following.

    \begin{lemma}
        \( T \) normalises \( \ssl \).
        \label{lem:normalises}
    \end{lemma}

    \begin{proof}
        We will show that the conjugate of each Bruhat element
        by \( T \) lies in \( \ssl \).

        Direct calculation shows that \( TwT^{-1}  = w \).
        For \( Tu_sT^{-1} \), observe that the only possible nonzero
        entry in the \( (2,1) \) block occurs in position
        \( (m,m) \) position of that block and is equal to the \( (m,m) \)
        entry of \( s \) which is \( 0 \) since \( s^* = -s \).
        Therefore the \( (2,1) \)-block
        of \( Tu_sT^{-1} \) is \( 0 \) and it follows immediately
        that \( Tu_sT^{-1}\in \ssl \).
				
			  Finally we claim that, for \( t \in A^\times \),
        \( Th_tT^{-1} \in \ssl \). 
        Let \( f \in M(m-1,\F ) \) be the matrix obtained by deleting the
        \( m \)th row and \( m \)th column of \( t \) and let $y$ be the
				 \( (m,m) \) entry of \( t^{*-1} \). Then the
				\( (1,1) \)-block of \( Th_tT^{-1}  \) is
        \[ a = \begin{pmatrix} f & 0 \\ 0 & y \end{pmatrix}. \]
				Since $t\in A^\times$, we have
        \( \rk(f) \geq m-2\). Moreover, by Cramer's rule, \(y= \det(t^*)^{-1} \det(f) \),
				which is
        \( 0 \) if and only if \( \rk(f) = m-2 \). It follows that if \( \rk(f) = m-1 \) then \( \rk(a) = m \)
        and if \( \rk(f) = m-2 \) then \( \rk(a) = m-2 \).
        In either case, the nullity of \( a \) is even and so \( Th_tT^{-1} \in \ssl \)
        by Lemma \ref{lem:evencorankcondition}.
    \end{proof}

    \begin{theorem}
        If \( m \) is even then
        \( \sll = \ssl \rtimes \langle T \rangle \)
        \label{thm:indexevencase}
    \end{theorem}

    \begin{proof}
        It suffices to show that for any
        \( X \in \sll  \) such that \( X \not\in \ssl \),
        we have \( TX \in \ssl \).
        By Lemma \ref{lem:normalises}, \( TX \in \ssl \)
        if and only if \( TGXH \in \ssl \)
        for any \( G,H \in \ssl \).
        In other words we can replace \( X \) by \( GXH \) where
        \( G,H \) are any elements of \( \ssl \).
        As in the proof of Theorem \ref{thm:indexoddcase}
        we can therefore assume that the
        \( (1,1) \)-block of \( X \) is \( k \) and that the \( (2,1) \)-block
        of \( X \) has zeroes in rows \(1,\dots, m-1 \) and has a nonzero
        \( (m,m) \)-entry. Now a straighforward computation shows that
        the \( (1,1) \)-block of \( TX \) is
        lower triangular with nonzero diagonal entries and is
        therefore nonsingular.
        So \( TX \in \ssl \) as required.
    \end{proof}

\section{On the index of \( \SSL^{\va}_*(2,A) \) in \( \SL^{\va}_*(2,A)\): the general case}\label{sec:general}

We continue to maintain the assumptions and notation from \S\ref{uno}.

\begin{theorem}\label{jopa} If $\r=0$ then $\SL^{-}_*(2,A)=\SSL^{-}_*(2,A)$.
\end{theorem}

 \begin{proof} This is shown in \cite[Corollary 3.4]{PS}.
\end{proof}

\begin{theorem}\label{jopa2} If $\r=0$,  $*$ is unramfied and $D=\F$ is a field then $\SL^{+}_*(2,A)=\SSL^{+}_*(2,A)$.
\end{theorem}

\begin{proof} By assumption there is a unit $u\in\F$ such that $u^*=-u$. Set
$$J_-=\begin{pmatrix} 0 & 1_m\\-1_m & 0 \end{pmatrix}, J_+=\begin{pmatrix} 0 & 1_m\\1_m & 0 \end{pmatrix}, K=\begin{pmatrix} 0 & u 1_m\\u 1_m & 0 \end{pmatrix},
P=\begin{pmatrix} 1_m & 0\\0 & u^{-1}1_m\end{pmatrix}
$$
and let $G_1, G_2$ and $G_3$ be the unitary groups respectively associated to $J_-, J_+$ and $K$. Observing that $G_2=G_3$ and
$P^*KP=J_-$, the map $\psi:G_1\to G_2$ given by $g\mapsto P^{-1}gP$ is a group isomorphism. Moreover, $\psi(h_t)=h_t$, $t\in A^\times$, $\psi(J_-)=J_+h_{-u}$ and
$\psi(u_s)=u_{su^{-1}}$, $s\in A$,  $s^*=s$. Furthermore, by Theorem \ref{jopa}, $G_1$ is generated by $J_-$ and all $h_t,u_s$,  $s^*=s$. We conclude that $G_2$ is generated
by $J_+$ and all $h_t,u_{r}$,  $r^*=-r$, as required.
\end{proof}

\begin{theorem}\label{jopa3} The reduction homomorphism $\Omega:\SL_*^\va(2,A)\to \SL_*^\va(2,\overline{A})$ is surjective.
\end{theorem}

\begin{proof} Clearly the canonical projection $A\to \overline{A}$ is surjective. We claim that the corresponding maps $A^\times\to  \overline{A}^\times$
and $A^{\varepsilon\text{-sym}}\to \overline{A}^{\varepsilon\text{-sym}}$ are also surjective. 
Indeed, if $x\in A$ and $\overline{x}\in \overline{A}^\times$ there is $y\in A$
such that $xy\equiv 1_m\equiv yx\mod M(m,\r)$. It is well known that 
\( M(m,\r) = J(M(m,B)) \),  therefore $x\in A^\times$. Moreover, if $s\in A$ and $\overline{s}\in \overline{A}^{\varepsilon\text{-sym}}$ then $s+\va s^*\in M(m,\r)$.
Set $s_0=(s-\va s^*)/2$. Clearly $s_0\in A^{\varepsilon\text{-sym}}$. Furthermore, since $s=(s-\va s^*)/2+(s+\va s^*)/2$, we have $\overline{s_0}=\overline{s}$.

The above claims and Theorems \ref{thm:indexoddcase}, \ref{thm:indexevencase}, \ref{jopa} and \ref{jopa2} show that all generators of $\SL_*^\va(2,\overline{A})$
can be reached.
\end{proof}

\begin{theorem}\label{jopa4} We have $\SL_*^\va(2,A)=\SSL_*^\va(2,A)$, except when $\va=1$, $*$ is ramified and $m$ is even, in which case
$\SL_*^\va(2,A)=\SSL_*^\va(2,A)\rtimes \langle T\rangle$, where $\SSL_*^\va(2,A)$ is the $\Omega$-preimage of the special orthogonal group $\mathrm{SO}_{2m}(\F)$.
\end{theorem}

\begin{proof} Note first of all that $\ker(\Omega)\subset\SSL_*^\va(2,A)$.  Indeed, let $X\in \ker(\Omega)$.
Then $X$ is an element of $\SL_*^\va(2,A)$ of the form $X=1_{2m}+Y$, where $Y\in M(2m,\r)$. The (1,1) block of $Y$ equals $1_m+Z$, where  $Z\in M(m,\r)=J(M(m,B))$,
so the (1,1) block $X$ is in $A^\times$ and therefore $X\in \SSL_*^\va(2,A)$.

Combining Theorems \ref{thm:indexoddcase}, \ref{jopa} and \ref{jopa2} and \ref{jopa3} with $\ker(\Omega)\subset\SSL_*^\va(2,A)$, we see that $\SL_*^\va(2,A)=\SSL_*^\va(2,A)$, except when $\va=1$, $*$ is ramified and $m$ is even. In this case, Theorems \ref{thm:indexevencase} and \ref{jopa3} with $\ker(\Omega)\subset\SSL_*^+(2,A)$ yield
$\SL_*^+(2,A)=\SSL_*^+(2,A)\rtimes \langle T\rangle$. Now $\Omega(\SSL_*^+(2,A))\subseteq \mathrm{SO}_{2m}(\F)$ and since
$[\SL_*^\va(2,A):\SSL_*^\va(2,A)]=2=[\mathrm{O}_{2m}(\F):\mathrm{SO}_{2m}(\F)]$, the kernel of $\SL_*^+(2,A)\to \mathrm{O}_{2m}(\F)/\mathrm{SO}_{2m}(\F)$
is precisely $\SSL_*^+(2,A)$.
\end{proof}

It is shown in \cite{CS} that if $|D|>3$ then (R1)-(R6) are defining relations for $\SL_*^{-}(2,A)$. On the other hand, Example \ref{notlocal} shows that if $B$ is not local then $\SL_*^{-}(2,A)$
need not be generated by its Bruhat elements.

\section{Basic Assumptions}\label{rifo}

%All rings in this paper are assumed to have an identity element different from zero. The unit group of a ring $B$ will be denoted by $B^\times$.

We henceforth fix $\va$ in $\{1,-1\}$ and a ring $B$, not necessarily commutative, subject to the following assumptions:

(A1) $B$ is finite.

(A2) $2\in B^\times$.

(A3) There is an involution $*$ on $B$.

(A4) There is a group homomorphism $\beta:B^+\to\C^\times$ that is primitive, in the sense that its kernel contains no right ideals of $B$ except $(0)$.

(A5) There is a nonzero finite right $B$-module $V$ and a hermitian or skew hermitian form $h:V\times V\to B$ relative to $*$. Moreover, $h$ is assumed to be nondegenerate,
in the sense that $h(u,V)=0$ implies $u=0$.

(A6) $\beta(b+\va b^*)=1$ for all $b\in B$.

\section{The Schr$\mathrm{\ddot{o}}$dinger representation}\label{schro}

Let $\U$ stand for the unitary group associated to $h$, that is,
$$
\U=\{x\in\GL(V)\,|\, h(xu,xv)=h(u,v)\text{ for all }u,v\in V\}.
$$
The Heisenberg group $H$ associated to $h$ has underlying set $B\times V$ and multiplication
$$
(b,u)(c,v)=(b+c+h(u,v),u+v).
$$
We have an action of $\U$ on $H$ via automorphisms as follows
$$
{}^g (b,u)=(b,gu).
$$
We identify the central subgroup $(B,0)$ of $H$ with $B^+$.

Given a $B$-submodule $N$ of $V$ we set
$$
N^\perp=\{u\in V\,|\, h(u,v)=0\text{ for all }v\in N\},\; N^\dagger=\{u\in V\,|\, \beta(2h(u,v))=1\text{ for all }v\in N\}.
$$
\begin{lemma}\label{da} If $N$ is a $B$-submodule of $V$ then $N^\perp=N^\dagger$.
\end{lemma}

\begin{proof} It is clear that $N^\perp\subseteq N^\dagger$. Suppose $u\in N^\dagger$.  Then the set $\{2h(u,v)\,|\, v\in N\}$ is a right ideal of $B$ contained in the kernel of~$\beta$,
so $2h(u,v)=0$ for all $v\in N$  by (A4). It follows from (A2) that $u\in N^\perp$, as required.
\end{proof}

\begin{lemma}\label{cxs} Let $N$ be any $B$-submodule of $V$. Then
$$
|V|=|N||N^\perp|.
$$
\end{lemma}

\begin{proof} Use (A5) and (A6) to mimic the proof of \cite[Lemma 2.1]{CMS2}.
\end{proof}

\begin{theorem} There is one and only one irreducible character, say $\chi_\beta$, of $H$ lying over~$\beta$. In particular, $\chi_\beta$ is $\U$-invariant.
\end{theorem}

\begin{proof} Let $M$ be a $B$-submodule of $V$ that is maximal relative to $\beta(2h(M,M))=0$. By definition, $M\subseteq M^\dagger$. However, (A6) yields
$$
\beta(2h(u,u))=1,\quad u\in V,
$$
so the maximality of $M$ implies that $M=M^\dagger$.

Consider the normal abelian subgroup $(B,M)$ of $H$ and extend $\beta$ to a group homomorphism $\beta_0:(B,M)\to\C^\times$ by
$$
\beta_0(b,u)=\beta(b).
$$
By definition, the stabilizer of $\beta_0$ in $H$, say $S_{\beta_0}$, consists of all $(c,v)\in H$ such that
$$
\beta(h(v,u)-\va h(v,u)^*)=1,\quad u\in M.
$$
Due to (A6) this translates into
$$
\beta(2h(v,M))=1.
$$
This means that $S_{\beta_0}=(B,M^\dagger)=(B,M)$. By Clifford theory,
$$
\chi_\beta=\mathrm{ind}_{(B,M)}^H \beta_0
$$
is an irreducible character of $H$ satisfying
$$
\mathrm{res}_{B^+}^H \chi_\beta=\frac{|V|}{|M|}\beta.
$$
It follows by Frobenius reciprocity that
\begin{equation}\label{fro1}
(\chi_\beta, \mathrm{ind}_{B^+}^H  \beta)=\frac{|V|}{|M|}.
\end{equation}
By Lemma \ref{da}, we have $M^\perp=M^\dagger=M$, whence $|M|^2=|V|$ by Lemma \ref{cxs}.
Therefore
\begin{equation}\label{fro2}
\deg \mathrm{ind}_{B^+}^H  \beta =|V|=\deg \frac{|V|}{|M|} \chi_\beta.
\end{equation}
Combining (\ref{fro1}) and (\ref{fro2}) we obtain
$$
\mathrm{ind}_{B^+}^H \beta= \frac{|V|}{|M|} \chi_\beta.
$$
By Frobenius reciprocity, $\chi_\beta$ is the only irreducible character of $H$ lying over $\beta$.
\end{proof}

\section{The Weil representation}\label{formula}

Given $B$-submodules $M,N$ of $V$ we consider the subgroups $\U_M$, $\U_{M,N}$ and $U^M$ of $U$, defined as follows:
$$
\U_M=\{g\in\U\,|\, gM=M\},
$$
$$
\U_{M,N}=\U_M\cap \U_N=\{g\in\U\,|\, gM=M\text{ and }gN=N\},
$$
$$
\U^M=\{g\in\U\,|\, gu=u\text{ for all }u\in M\}.
$$

We make the following general assumption:

(A7) There exist $B$-submodules $M,N$ of $V$ such that $V=M\oplus N$ and
$$
h(M,M)=0=h(N,N).
$$

We will keep $M,N$ for the remainder of the paper. It is clear from (A5) and (A7) that $M=M^\perp$. Thus Lemma \ref{da} yields $M=M^\dagger$,
whence $M$ is maximal subject to $\beta(2 h(M,M))=1$.

We next construct an irreducible $H$-module affording $\chi_\beta$. Let $Y=\C y$ be a one dimensional $(B,M)$-module
affording $\beta_0$, so that
$$
(b,u)\cdot y=\beta(b)\cdot y,\quad b\in B,u\in M.
$$
Then $H$ acts on $\ind_{(B,M)}^{H} Y$ with character $\chi_\beta$. For our purposes, it will be convenient to make $Y$ into a
module over $(B,M)\rtimes \U_M$ via
$$
(b,u)g\cdot y =\beta(b)\cdot y, \quad b\in B,u\in M, g\in\U_M,
$$
and to replace $\ind_{(B,M)}^{H} Y$ by
$$
X=\ind_{(B,M)\rtimes \U_M}^{H\rtimes \U_M} Y.
$$
It is clear that $H$ also acts on $X$ with character $\chi_\beta$. Let $S:H\to \GL(X)$ be the representations arising from
the action of $H$  on $X$. Since $\chi_\beta$ is $\U$-invariant, given any $g\in\U$ there is a unique operator $P(g)\in\GL(X)$, up to scaling, such that
\begin{equation}
\label{inter}
P(g)S(h) P(g)^{-1}=S({}^g h),\quad h\in H.
\end{equation}
Let $P:\U_M\to \GL(X)$ be the representation arising from the action of $\U_M$ on $X$. Clearly, given any $g\in\U_M$, the operator $P(g)$ satisfies (\ref{inter}).
We have a basis $(e_v)_{v\in N}$ of $X$, where $e_v=(0,v)y\in X$ for $v\in N$. Direct calculation shows that
\begin{equation}
\label{gza}
S(0,w)e_v=e_{v+w},\quad v,w\in N,
\end{equation}
\begin{equation}
\label{gzb}
S(0,u)e_v=\beta(2h(u,v))e_{v},\quad u\in M,v\in N.
\end{equation}
\begin{equation}
\label{gzc}
S(b,0)e_v=\beta(b)e_{v},\quad b\in B,v\in N,
\end{equation}
\begin{equation}
\label{gzd}
P(g)e_v=e_{gv},\quad g\in \U_{M,N},v\in N,
\end{equation}
\begin{equation}
\label{gze}
P(g)e_v=\beta(h(gv,v))e_{v},\quad g\in \U^M,v\in N.
\end{equation}
As a special case of (\ref{gzd}), we have
$$
P(-1_V)e_v=e_{-v},\quad v\in N.
$$
Let $X_{\pm}$ be the eigenspaces of $X$ with eigenvalues $\pm 1$ for $P(-1_V)$. Let $I$ be a subset of $N\setminus\{0\}$ obtained by
selecting one and only one element out of $\{v,-v\}$ for every $v\in N\setminus\{0\}$. Then $e_0$ and $e_{v}+e_{-v}$, with $v\in I$,
form a basis of $X_+$ and $e_{v}-e_{-v}$, with $v\in I$, form a basis of $X_{-}$.

Let $P:\U\to\GL(X)$ be any function satisfying (\ref{inter}) and extending the group homomorphism $P:\U_M\to\GL(X)$ defined above.

\begin{theorem}\label{weilrep} The subspaces $X_\pm$ are invariant under all $P(g)$, $g\in\U$. Moreover, let
$$
c(g)=(\det P(g)|_{X_+})^{-1}\det P(g)|_{X_-}\in\C^\times, \quad g\in \U.
$$
Then $W:\U\to\GL(X)$, given by $W(g)=P(g)c(g)$, is a representation (called Weil representation of type $\beta$).
\end{theorem}

\begin{proof} This follows as in \cite[\S3]{CMS}.
\end{proof}

It is clear that for all $g\in\U$, we have
\begin{equation}
\label{interw}
W(g)S(h) W(g)^{-1}=S({}^g h),\quad h\in H.
\end{equation}
A Weil representation is uniquely determined by (\ref{interw}) up to a linear character of $\U$.

From (\ref{gzd}), (\ref{gze}) and the definition of $c(g)$  we easily find that
\begin{equation}
\label{f2}
W(g)e_v=\beta(h(gv,v)) e_{v},\quad v\in N,g\in\U^M
\end{equation}
and
\begin{equation}
\label{f1G}
W(g)e_v=(-1)^{|I_g|}e_{gv},\quad v\in N, g\in \U_{M,N},
\end{equation}
where
$$
I_g=\{v\in I\,|\, gv\in -I\}.
$$
Since $P$ and $W$ are group homomorphisms from $\U_{M,N}$, it follows that so is $c$. In other words, the function
$$
\mu:\U_{M,N}\to \{-1,1\}\subset\C^\times,\quad
$$
defined by
\begin{equation}
\label{anj}
\mu(g)=(-1)^{|I_g|},\quad g\in \U_{M,N},
\end{equation}
is a group homomorphism independent of the choice of $I$.

By the ``free case" we understand the case when $M$ has a basis $\{u_1,\dots, u_m\}$, $N$ has a basis $\{v_1,\dots,v_m\}$,
and the Gram matrix of $h$ relative to the basis $\BB=\{u_1,\dots, u_m,v_1,\dots,v_m\}$ of $V$ is
$$ J = \begin{pmatrix}0 & 1_m \\ \va 1_m & 0
\end{pmatrix}.
$$

We next translate the above into matrix form, within the free case. Set $A=M(m,B)$ and denote also by $*$ the involution that $A$
inherits from $B$, as indicated in \S\ref{uno}. Recall the definition of the Bruhat elements $\omega, h_t,u_r\in \SL_*^{\varepsilon}(2,A)$
and the fact that the map $\U\to \SL^\va_*(2,A)$, given by $g\mapsto M_{\BB}(g)$, is a group isomorphism. Under this isomorphism, $\U_{M,N}$
(resp. $\U^M$) corresponds to the subgroup of $\SL_*^{\varepsilon}(2,A)$ of all $h_t$, $t\in A^\times$ (resp. all $u_r$, $r\in A^{\varepsilon\text{-sym}}$). It follows that $\U_M=\U^M\rtimes \U_{M,N}$.

Let $X$ be a complex vector space with basis $(e_a)_{a\in B^m}$. The above gives the representation $W:\SL_*^{\varepsilon}(2,A)\to\GL(X)$,
where
\begin{equation}
\label{g2}
W(u_S)e_a=\beta(-\va a^*Sa) e_{a},\quad a\in B^m,S\in A^{\varepsilon\text{-sym}},
\end{equation}
\begin{equation}
\label{g1G}
W(h_T)e_a=\mu(T^*)e_{(T^*)^{-1}a},\quad a\in B^m, T\in A^\times,
\end{equation}
where $\mu:\GL(m,B)\to\{-1,1\}$ is the matrix analogue of group homomorphism defined by (\ref{anj}), where now $N$ is replaced by the column space
$B^m$. We have used the obvious fact that $\mu(T^{-1})=\mu(T)$ and we will see later that $\mu(T^*)=\mu(T)$.

\section{Computing $P$ and $c$}\label{ela}

We go back to the general case. Let $g\in\U$ and suppose $P(g)\in\GL(X)$ satisfies (\ref{inter}). Then by (\ref{inter}) and (\ref{gza})
\begin{equation}
\label{pu}
P(g)e_v=P(g)S(0,v)e_0=S^g(0,v)P(g)e_0=S(0,gv)P(g)e_0,\quad v\in N.
\end{equation}
This says that $P(g)$ is uniquely determined by $x_0=P(g)e_0\neq 0$. To find $x_0$ note that (\ref{inter}) and (\ref{gzb}) give
$$
S^g(0,u)x_0=S^g(0,u)P(g)e_0=P(g)S(0,u)e_0=P(g)e_0=x_0.
$$
This says that $x_0$ is a fixed point of all operators $S^g(0,u)$, $u\in M$. Since the space of points that is fixed by all
$S(0,u)$, $u\in M$, is $\C e_0$ (use that $h$ is nondegenerate and $\beta$ is primitive) and $S$ and $S^g$ are similar, we see
that the space of points that is fixed by all $S^g(0,u)$, $u\in M$, is also one dimensional and equal to $\C x_0$.
This gives a unique way, up to scaling, to determine $P(g)$: by finding a (nonzero) fixed point of all $S^g(0,u)$, $u\in M$.

We make a new general assumption:

(A8) There exists $z\in\U$ such that $z(M)=N$ and $z(N)=M$,

\noindent and apply the foregoing ideas to the case of $g=z$. It is clear that
$$
x_0=\underset{w\in N}\sum e_w
$$
is a fixed point of all $S(0,v)$, $v\in N$, that is, all $S(0,zu)$, $u\in M$. Using this fixed point, (\ref{gzb}) and (\ref{pu}), we obtain
\begin{equation}
\label{pu2}
P(z)e_v=\underset{w\in N}\sum \beta(2h(zv,w))e_w,\quad v\in N.
\end{equation}

We illustrate this in the free case by obtaining a matrix version of (\ref{pu2}) for special values of $z$. Consider
the unitary transformations $s,t\in\U$, respectively defined by
$$
u_i\mapsto v_i,\; v_i\mapsto \va u_i,\quad 1\leq i\leq m,
$$
$$
u_i\mapsto \va v_i,\; v_i\mapsto u_i,\quad 1\leq i\leq m.
$$
It is clear that both satisfy (A8). Given $v\in N$, we have
$$
v=v_1h(u_1,v)+\cdots+v_m h(u_m,v)
$$
and therefore
$$
sv=\va(u_1h(u_1,v)+\cdots+u_m h(u_m,v)).
$$
Hence
$$
\begin{aligned}
h(sv,w)& =h(\va(u_1h(u_1,v)+\cdots+u_m h(u_m,v)),v_1h(u_1,w)+\cdots+v_m h(u_m,w))\\
&=\va(h(u_1,v)^*h(u_1,w)+\cdots+h(u_m,v)^*h(u_m,w)).
\end{aligned}
$$
Thus
\begin{equation}
\label{pu3}
P(s)e_v=\underset{w\in N}\sum \beta(\va\; 2 \underset{1\leq i\leq m}\sum h(u_i,v)^*h(u_i,w)),\quad v\in N.
\end{equation}
In matrix form, this gives
\begin{equation}
\label{pu4}
P(s)e_a=\underset{b\in B^m}\sum \beta(\va\; 2a^* b)e_b,\quad a\in B^m.
\end{equation}
Likewise, since
$$
tv=u_1h(u_1,v)+\cdots+u_m h(u_m,v),
$$
we have
\begin{equation}
\label{tu3}
P(t)e_v=\underset{w\in N}\sum \beta(2 \underset{1\leq i\leq m}\sum h(u_i,v)^*h(u_i,w)),\quad v\in N
\end{equation}
and
\begin{equation}
\label{tu4}
P(t)e_a=\underset{b\in B^m}\sum \beta(2a^* b)e_b,\quad a\in B^m.
\end{equation}

Going back to the general case, we next compute $c(z)$ up to a sign (a more precise calculation is given below). To this end
we make a new general assumption:

(A9) $z^2=\va 1_V$.

This is certainly satisfied in the free case with $z=s$ and $z=t$.

Let us deal first with the case $\va=-1$. Due to (\ref{inter}) and Schur's Lemma, $P(z)^2$ and $W(-1_V)$ differ by a scalar multiple. Entry (1,1) (with respect to the basis $e_v$, $v\in N$) of the matrix of $W(-1_V)$ is equal to $(-1)^{\frac{|N|-1}{2}}$. On the other hand, entry $(1,1)$ of the matrix of $P(z)^2$ is equal to $|N|$. Now $W(z)=c(z)P(z)$, so $W(-1_V)=c(z)^2 P(z)^2$. It follows that
$$
(-1)^{\frac{|N|-1}{2}}=c(z)^2 |N|,
$$
whence
$$
c(z)^2=\frac{(-1)^{\frac{|N|-1}{2}}}{|N|}.
$$
This determines $c(z)$ up to a sign. Note that in the free case, we have $|N|=|B|^m$, so
$$
c(z)^2=\frac{(-1)^{\frac{|N|-1}{2}}}{|B|^m}.
$$

In the case $\va=1$, $P(z)^2$ and $W(1_V)$ differ by a scalar multiple, and the same calculation as above shows that
$$
c(z)^2=\frac{1}{|N|},
$$
which in the free case becomes
$$
c(z)^2=\frac{1}{|B|^m}.
$$

\section{Further calculations involving $P$ and $c$}\label{fela}

From now on we work exclusively in the free case.

Let $T\in\GL(m,B)$ and suppose $T^*+\va T=0$. The existence of such a $T$ is clear in the skew hermitian case, as well as in the hermitian case when $m$ is even.
No such a $T$ exists in the hermitian case when $m$ is odd in the ramified even case (as defined in \S\ref{exam}).

Associated to $T$ we have the elements $g_T,k_T,\ell_T\in \U$ whose matrices relative to $\BB$ are respectively equal to
$$
\left(
    \begin{array}{cc}
      1_m & 0 \\
      T & 1_m \\
    \end{array}
  \right), \left(
    \begin{array}{cc}
      1_m & T \\
      0 & 1_m \\
    \end{array}
  \right), \left(
    \begin{array}{cc}
      0 & -T^{-1} \\
      T & 0 \\
    \end{array}
  \right).
$$
Note that we have
$$
\left(
    \begin{array}{cc}
      1_m & -T^{-1} \\
      0 & 1_m \\
    \end{array}
  \right)\left(\begin{array}{cc}
      1_m & 0 \\
      T & 1_m \\
    \end{array}
  \right)\left(
    \begin{array}{cc}
      1_m & -T^{-1} \\
      0 & 1_m \\
    \end{array}
  \right)= \left(
    \begin{array}{cc}
      0 & -T^{-1} \\
      T & 0 \\
    \end{array}
  \right),
$$
that is
\begin{equation}
\label{ga3}
k_{-T^{-1}}g_T k_{-T^{-1}}=\ell_T.
\end{equation}

    According to (\ref{f2}), we have
\begin{equation}
\label{v6}
W(k_{-T^{-1}})e_v=\beta(h(k_{-T^{-1}}v,v)) e_{v},\quad v\in N.
\end{equation}

We next find $P(g)$ for $g=g_T$. By the general method outlined in \S\ref{ela}, to determine $P(g)$ we need to find a nonzero $x_0$ in the space of points fixed by all $S^g(0,u)$, $u\in M$.

To this end, note first of all that
$$
S^g(0,w)=S(0,gw)=S(0,w),\quad w\in N,
$$
which implies
$$
P(g)S(0,w)P(g)^{-1}=S^g(0,w)=S(0,w),\quad w\in N.
$$
This means that $P(g)$ commutes with all operators $S(0,w)$, $w\in N$.

The restriction of $g-1_V$ to $M$ defines an isomorphism of $B$-modules, $M\to N$, whose inverse is
the restriction of $k_{T^{-1}}-1_V$ to $N$. Thus, for each $w\in N$ there is a unique element $u^w=(k_{T^{-1}}-1_V)w\in M$ such that $gu^w-u^w=w$.
We claim that the nonzero element
$$
x_0=\underset{w\in N}\sum\beta(h(u^w,w))e_w
$$
is a fixed point of all $S^g(0,u)$, $u\in M$. In order to verify the claim, we first note that
\begin{equation}
\label{ne}
\beta(h(x,gy-y))=\beta(h(y,gx-x)),\quad x,y\in V.
\end{equation}
Indeed, since $gx-x\in N$ for all $x\in V$ and $h(N,N)=0$, making use of (A6) we obtain
$$
\begin{aligned}
\beta(h(x,gy-y)) &= \beta(h(x-gx+gx,gy-y))=\beta(h(gx,gy-y))=\beta(h(x,y)-h(gx,y))\\
&=\beta(-h(y,x)+h(y,gx))=\beta(h(y,gx-x)).
\end{aligned}
$$
Now, using (A6) once more and appealing to (\ref{gzc}) we see that
\begin{equation}
\label{v1}
S^g(0,u)=S(0,gu)=S(0,gu-u)S(0,u)\beta(h(u,gu)),\quad u\in M.
\end{equation}
Thus, (\ref{gza}), (\ref{gzb}) and (\ref{v1}) yield
\begin{equation}
\label{v1t}
S^g(0,u)x_0=\underset{w\in N}\sum\beta(2h(u,w)+h(u,gu)+h(u^w,w))e_{w+gu-u},\quad u\in M.
\end{equation}
Make the change of variable $w'=w+gu-u$ and expand (\ref{v1t}) accordingly. By definition, we have $u^{w'-(gu-u)}=u^{w'}-u$.
Using (\ref{ne}) and (A6) once again we see that  (\ref{v1t}) is transformed into
$$
S^g(0,u)x_0=\underset{w'\in N}\sum\beta(h(u^{w'},w'))e_{w'}=x_0,\quad u\in M.
$$
This proves the claim.

Since $P(g)$ commutes with all $S(0,w)$, $w\in N$, (\ref{gza}) gives
$$P(g)e_v=P(g)S(0,v)e_0=S(0,v)P(g)e_0=S(0,v)x_0, \quad v\in N.$$
This and the above claim yield
$$
P(g)e_v=\underset{w\in N}\sum\beta(h(u^w,w))e_{v+w},\quad v\in N.
$$
Since $h(N,N)=0$, $h((k_{T^{-1}}-1_V)w,w)=h(k_{T^{-1}}w,w)$ for all $w\in N$, whence
$$
P(g)e_v=\underset{w\in N}\sum\beta(h(k_{T^{-1}}w,w))e_{v+w},\quad v\in N,
$$
or
\begin{equation}\label{defin}
P(g)e_v=\underset{w\in N}\sum\beta(h(k_{T^{-1}}(w-v),w-v))e_{w},\quad v\in N.
\end{equation}
This defines an intertwining operator between $S$ and $S^g$.

We next claim that
\begin{equation}\label{vg}
c(g)=\big(\underset{w\in N}\sum\beta(h( k_{T^{-1}}w,w))\big)^{-1}.
\end{equation}
Indeed, let $\widehat{N}$ be the group of all group homomorphisms $N\to\C^\times$. For each $\nu\in \widehat{N}$
set
$$
X^\nu=\{x\in X\,|\, S(0,w)=\nu(w)x\text{ for all }w\in N\}
$$
and
$$
y_\nu=\underset{w\in N}\sum \nu(w)^{-1}e_w\in X.
$$
It is well-known and easy to see that
$$\underset{\nu\in \widehat{N}}\sum X^\nu=\underset{\nu\in \widehat{N}}\bigoplus X^\nu,$$
and
$$
y_\nu\in X^{\nu}.
$$
Since $|\widehat{N}|=|N|$, it follows that each $X^\nu$ is one dimensional and spanned by $y_{\nu}$.
Now $P(g)$ commutes with all $S(0,w)$, $w\in N$, so $P(g)$ preserves each $X^\nu$, $\nu\in \widehat{N}$. This means
that each $y_\nu$ is an eigenvector for $P(g)$. Let $l_\nu^g$ be the eigenvalue for $P(g)$ acting on $y_\nu$. Note that
$$
P(-1_V)y_\nu=y_{{\nu}^{-1}},\quad \nu\in \widehat{N},
$$
We have already mentioned that $P(g)P(-1_V)=P(-1_V)P(g)$ (see \cite[\S 3]{CMS}). Therefore,
$$
P(g)y_{{\nu}^{-1}}=P(g)P(-1_V)y_\nu=P(-1_V)P(g)y_\nu=P(-1_V)l^g_\nu y_\nu=l^g_\nu y_{{\nu}^{-1}},
$$
which implies that
$$
l^g_\nu=l^g_{{\nu}^{-1}},\quad \nu\in \widehat{N}.
$$
It follows from the definition of $c(g)$ that
$$
c(g)=(l_1^g)^{-1}.
$$
Here $P(g)y_1=l_1^g y_1$, that is
$$
P(g)(\underset{w\in N}\sum e_w)=l_1^g(\underset{w\in N}\sum e_w).
$$
Using the definition of $P(g)$ this forces
$$
l_1^g=\underset{w\in N}\sum\beta(h(k_{T^{-1}} w,w)).
$$
Note that since $l_1^g$ is an eigenvalue of the invertible operator $P(g)$, the right hand side of the last equation is not 0. This proves (\ref{vg}). We note at this point that
$$
\underset{w\in N}\sum\beta(h( k_{T^{-1}} w,w))=\underset{b\in B^m}\sum \beta(b^* (-\va T^{-1}) b),
$$
so
\begin{equation}\label{vg2}
c(g_T)=\big(\underset{b\in B^m}\sum\beta(b^* (-\va T^{-1}) b)\big)^{-1}.
\end{equation}
From (\ref{defin}) and (\ref{vg2}) one gets
\begin{equation}
\label{j22}
W(g_T)e_v=\big(\underset{b\in B^m}\sum\beta(b^* (-\va T^{-1}) b)\big)^{-1}
\underset{w\in N}\sum\beta(k_{T^{-1}}(w-v),w-v))e_{w},\quad v\in N.
\end{equation}
Combining (\ref{ga3}), (\ref{v6}) and (\ref{j22}), and using (A6) as well as $T^*+\va T=0$ we obtain
\begin{equation}
\label{j23}
W(\ell_T)e_v=\big(\underset{b\in B^m}\sum\beta(b^* (-\va T^{-1}) b)\big)^{-1}\underset{w\in N}\sum\beta(-2 h(
 k_{T^{-1}}v,w))e_{w},\quad v\in N.
\end{equation}

For $T\in\GL(m,B)$, let $f_T$ be the element of $\U_{M,N}$ such that
$$
M_{\BB}(f_T)=h_T=\left(
    \begin{array}{cc}
      T & 0 \\
      0 & (T^*)^{-1} \\
    \end{array}
  \right).
$$

\section{The skew hermitian case}\label{skewsec}

In this section we assume $\va=-1$. Obviously $1_m\in\GL(m,B)$ and $1_m^*=1_m$.

We next specialize (\ref{j23}) to the important case $s=\ell_{1_m}$. For $v\in N$, we have
$$
v=v_1a_1+\cdots+v_m a_m,\quad a_i\in B,
$$
so
$$
(k_{1_m}-1_V)v=u_1a_1+\cdots+u_m a_m=-sv.
$$
Thus by (\ref{j23})
\begin{equation}
\label{j27}
W(s)e_v=\big(\underset{b\in B^m}\sum\beta(b^* b)\big)^{-1}\underset{w\in N}\sum\beta(2 h(sv,w))e_{w},\quad v\in N,
\end{equation}
which is in perfect agreement with (\ref{pu2}).

Next set
$$
\widehat{G}(\beta)=\underset{b\in B}\sum\beta(b^* b).
$$
Then, we clearly have
\begin{equation}
\label{j28}
\underset{b\in B^m}\sum\beta(b^* b)=\widehat{G}(\beta)^m.
\end{equation}
Combining (\ref{j27}) and (\ref{j28}) we obtain
\begin{equation}
\label{j29}
W(s)e_v=\widehat{G}(\beta)^{-m}\underset{w\in N}\sum\beta(2 h(sv,w))e_{w},\quad v\in N,
\end{equation}
Using (\ref{pu3}) and (\ref{pu4}) we also find
\begin{equation}
\label{mu3}
W(s)e_v=\widehat{G}(\beta)^{-m}\underset{w\in N}\sum \beta(-2 \underset{1\leq i\leq m}\sum h(u_i,v)^*h(u_i,w))e_w,\quad v\in N,
\end{equation}
\begin{equation}
\label{mu4}
W(s)e_a=\widehat{G}(\beta)^{-m}\underset{b\in B^m}\sum \beta(-2a^* b)e_b,\quad a\in B^m.
\end{equation}
In the case of $t=sf_{-1_m}$, we deduce
\begin{equation}
\label{mu5}
W(t)e_v=\widehat{G}(\beta)^{-m}(-1)^{\frac{|N|-1}{2}}\underset{w\in N}\sum \beta(2 \underset{1\leq i\leq m}\sum h(u_i,v)^*h(u_i,w))e_w,\quad v\in N
\end{equation}
and
\begin{equation}
\label{mu6}
W(t)e_a=\widehat{G}(\beta)^{-m}(-1)^{\frac{|N|-1}{2}}\underset{b\in B^m}\sum \beta(2a^* b)e_b,\quad a\in B^m.
\end{equation}

Observe next that for $T\in\GL(m,B)$ satisfying $T^*=T$, we have
$$
 \left(
    \begin{array}{cc}
      T^{-1} & 0 \\
      0 & T\\
    \end{array}
  \right) \left(
    \begin{array}{cc}
      0 & -1_m \\
      1_m & 0 \\
    \end{array}
  \right)=\left(
    \begin{array}{cc}
      0 & -T^{-1} \\
      T & 0 \\
    \end{array}
  \right),
$$
that is
\begin{equation}
\label{ka3}
f_{T^{-1}}s=\ell_T.
\end{equation}

If we now use (\ref{f1G}), (\ref{anj}), (\ref{j29}) and (\ref{ka3}) we obtain
$$
W(\ell_T)e_v=\mu(T) \widehat{G}(\beta)^{-m}\underset{w\in N}\sum\beta(2 h(sv,w))e_{f_{T^{-1}}w},\quad v\in N,
$$
$$
W(\ell_T)e_v=\mu(T) \widehat{G}(\beta)^{-m}\underset{w\in N}\sum\beta(2 h(sv,f_T w))e_{w},\quad v\in N,
$$
$$
W(\ell_T)e_v=\mu(T) \widehat{G}(\beta)^{-m}\underset{w\in N}\sum\beta(2 h(f_{T^{-1}}sv,w))e_{w},\quad v\in N,
$$
\begin{equation}
\label{mu7}
W(\ell_T)e_v=\mu(T) \widehat{G}(\beta)^{-m}\underset{w\in N}\sum\beta(-2  h(k_{T^{-1}}v,w))e_{w},\quad v\in N.
\end{equation}

Comparing (\ref{j23}) and (\ref{mu7}) we obtain
$$
\underset{b\in B^m}\sum\beta(b^* T^{-1} b)=\mu(T) \underset{b\in B^m}\sum\beta(b^*b).
$$
Using that $\mu$ takes values in $\{1,-1\}$ as well the formulas at the end of \S\ref{ela} we deduce the following result.

\begin{theorem}\label{mu1} Suppose $T\in\GL(m,B)$ satisfies $T^*=T$. Then
\begin{equation}
\label{mu82}
\underset{b\in B^m}\sum\beta(b^* T b)=\underset{b\in B^m}\sum\beta(b^* T^{-1} b)=\mu(T) \underset{b\in B^m}\sum\beta(b^*b),
\end{equation}
\begin{equation}
\label{mu9}
\big(\underset{b \in B^m}\sum \beta(b^* T b)\big)^2=(-1)^{\frac{|N|-1}{2}}|N|.
\end{equation}
\end{theorem}

Next let $T\in\GL(m,B)$ be arbitrary. Using the identity
$$\left(
    \begin{array}{cc}
      0 & -1_m \\
      1_m & 0 \\
    \end{array}
  \right)\left(
    \begin{array}{cc}
      T & 0 \\
      0 & (T^*)^{-1} \\
    \end{array}
  \right)=\left(
    \begin{array}{cc}
       (T^*)^{-1} & 0 \\
      0 & T \\
    \end{array}
  \right)\left(
    \begin{array}{cc}
      0 & -1_m \\
      1_m & 0 \\
    \end{array}
  \right)
$$
and applying $W$ to both sides, using the given formulas for $W(s)$ and $W(h_T)$, we derive the following result.

\begin{theorem}\label{mu55} Suppose $T\in\GL(m,B)$. Then
\begin{equation}
\label{vamu}
\mu(T)=\mu(T^*).
\end{equation}
\end{theorem}

Note that (\ref{vamu}) is valid in the hermitian case as well, as this fact is independent of the form $h$.

Observe next that (\ref{vamu}) allows us to simplify (\ref{g1G}), regardless of the nature of $h$, as follows:
\begin{equation}
\label{g1G33}
W(h_T)e_a=\mu(T)e_{(T^*)^{-1}a},\quad a\in B^m, T\in \GL(m,B).
\end{equation}

In view of (\ref{g2}), (\ref{mu6}) and (\ref{g1G33}) we have the following result.

\begin{theorem}\label{weilrepskewh} The Weil representation $W:\SL^{-}_*(2,A)\to\GL(X)$, given by Theorem \ref{weilrep} through
the isomorphism $\U\to \SL^{-}_*(2,A)$, $g\mapsto M_{\BB}(g)$, is defined as follows on the Bruhat elements:
\begin{equation}
\label{fer1}
W(h_T)e_a=\mu(T)e_{(T^*)^{-1}a},\quad a\in B^m, T\in A^\times,
\end{equation}
\begin{equation}
\label{fer2}
W(u_S)e_a=\beta(a^*Sa) e_{a},\quad a\in B^m,S\in A, S^*=S,
\end{equation}
\begin{equation}
\label{fer3}
W(\omega)e_a=\widehat{G}(\beta)^{-m}(-1)^{\frac{|N|-1}{2}}\underset{b\in B^m}\sum \beta(2a^* b)e_b,\quad a\in B^m.
\end{equation}
\end{theorem}

\section{The hermitian case}\label{sech}

In this section we assume that $\va=1$ and $m=2n$ is even. Let
$$
Q=\left(
    \begin{array}{cc}
      0 & -1_n \\
      1_n & 0 \\
    \end{array}
  \right).
$$
Note that $Q\in\GL(m,B)$ and $Q^*=-Q=Q^{-1}$. A simple calculation, using (A6), shows that
$$
\underset{b\in B^m}\sum\beta(b^* Q b)=\underset{a,c\in B^n}\sum\beta(-2a^* c).
$$
Since $\beta$ is primitive, for $0\neq a\in B^n$, the linear character $B^n\to \C^\times$ given by $c\mapsto \beta(-2a^* c)$
is nontrivial, whence
$$
\underset{c\in B^n}\sum\beta(-2a^* c)=0
$$
and a fortiori
$$
\underset{b\in B^m}\sum\beta(b^* Q b)=\underset{c\in B^n}\sum\beta(0)=|B|^n.
$$
Therefore, by (\ref{j23})
\begin{equation}
\label{j23her}
W(\ell_Q)e_v=\frac{1}{|B|^n}\underset{w\in N}\sum\beta(-2 h(k_{Q^{-1}}v,w))e_{w},\quad v\in N.
\end{equation}
Next note that
$$
f_{-Q}\ell_Q=s.
$$
Now, since $|B|$ is odd, we see that
$$
\mu(-Q)=1.
$$
Thus, the formulas for $W(f_{-Q})$ and $W(\ell_Q)$ give
$$
\begin{aligned}
    W(s)e_v & =\frac{1}{|B|^n}\underset{w\in N}\sum\beta(-2 h(k_{Q^{-1}}v,w))e_{f_{-Q} w},\quad v\in N\\
    & =\frac{1}{|B|^n}\underset{w\in N}\sum\beta(-2 h(k_{Q^{-1}}v,f_{-Q}^{-1}  w))e_{w},\quad v\in N\\
    & =\frac{1}{|B|^n}\underset{w\in N}\sum\beta(-2 h(f_{-Q} k_{Q^{-1}}v,w))e_{w},\quad v\in N
\end{aligned}
$$
Now $-h(f_{-Q}k_{Q^{-1}}v,w) = h(sv,w)$ for all $v,w \in N$ and therefore
\begin{equation}
\label{pep}
W(s)e_v=\frac{1}{|B|^n}\underset{w\in N}\sum\beta(2 h(sv,w))e_{w},\quad v\in N.
\end{equation}
This is in agreement with (\ref{pu2}). In matrix terms, this says that
\begin{equation}
\label{ka44}
W(s)e_a=\frac{1}{|B|^n}\underset{b\in B^m}\sum\beta(2 a^* b)e_{b},\quad a\in B^m.
\end{equation}

In view of (\ref{g2}), (\ref{g1G33}) and (\ref{ka44}) we have the following result.

\begin{theorem}\label{weilrepher} The Weil representation $W:\SL^{+}_*(2,A)\to\GL(X)$, given by Theorem \ref{weilrep} through
the isomorphism $\U\to \SL^{+}_*(2,A)$, $g\mapsto M_{\BB}(g)$, is defined as follows on the Bruhat elements:
\begin{equation}
\label{ferna1}
W(h_T)e_a=\mu(T)e_{(T^*)^{-1}a},\quad a\in B^m, T\in A^\times,
\end{equation}
\begin{equation}
\label{ferna2}
W(u_S)e_a=\beta(-a^*Sa) e_{a},\quad a\in B^m,S\in A, S^*=-S,
\end{equation}
\begin{equation}
\label{ferna3}
W(\omega)e_a=\frac{1}{|B|^n}\underset{b\in B^m}\sum\beta(2 a^* b)e_{b},\quad a\in B^m.
\end{equation}
\end{theorem}

Suppose next that $T\in\GL(m,B)$ satisfies $T^*=-T$. Then
$$
 \left(
    \begin{array}{cc}
      -T^{-1} & 0 \\
      0 & T\\
    \end{array}
  \right) \left(
    \begin{array}{cc}
      0 & 1_m \\
      1_m & 0 \\
    \end{array}
  \right)=\left(
    \begin{array}{cc}
      0 & -T^{-1} \\
      T & 0 \\
    \end{array}
  \right),
$$
that is
\begin{equation}
\label{ka33}
f_{-T^{-1}}s=\ell_T.
\end{equation}

If we now use  (\ref{f1G}), (\ref{anj}), (\ref{pep}) and (\ref{ka33}) we obtain
$$
W(\ell_T)e_v=\mu(T) \frac{1}{|B|^n}\underset{w\in N}\sum\beta(2 h(sv,w))e_{f_{-T^{-1}}w},\quad v\in N,
$$
$$
W(\ell_T)e_v=\mu(T) \frac{1}{|B|^n}\underset{w\in N}\sum\beta(2 h(sv,f_{-T} w))e_{w},\quad v\in N,
$$
$$
W(\ell_T)e_v=\mu(T) \frac{1}{|B|^n}\underset{w\in N}\sum\beta(2 h(f_{-T^{-1}}sv,w))e_{w},\quad v\in N,
$$
\begin{equation}
\label{mu77}
W(\ell_T)e_v=\mu(T) \frac{1}{|B|^n}\underset{w\in N}\sum\beta(-2 h(k_{T^{-1}}v,w))e_{w},\quad v\in N.
\end{equation}

Comparing (\ref{j23}) and (\ref{mu77}) one gets
$$
\underset{b\in B^m}\sum\beta(b^* (-T^{-1}) b)=\mu(T) |B|^n.
$$
But $\mu(T)=\mu(T^{-1})$ and $\mu(-1_m)=1$ since $|N|$ is square. This gives the following result.
\begin{theorem}\label{mu2} Suppose that $T\in\GL(m,B)$ satisfies $T^*=-T$. Then
\begin{equation}
\label{mu89}
\underset{b\in B^m}\sum\beta(b^* T b)=\mu(-T^{-1}) |B|^n=\mu(T) |B|^n.
\end{equation}
\end{theorem}

%This implies, in particular, that
%$$
%\big(\underset{b\in B^m}\sum\beta(b^* T b)\big)^2=|N|.
%$$

%Note also that $\mu(-1_m)=1$ because $|N|$ is a square. All of this information is required in the next section.

\section{Comparison}\label{comop}

Given a ring $A$ endowed with an involution $*$, \cite{GPS} provides  an abstract procedure to construct a Weil representation of $\SL_*^{\va}(2, A)$  when this group has a Bruhat presentation. For our purposes, we require the following adaptation of this method
(besides the required changes from the right to left $A$-module point of view, we have changed { (\ref{gamma3})} and, accordingly, (\ref{c})).

A data for $(A,\va)$ is a 5-tuple  $(P,\chi,\gamma, \alpha,f)$
where
\begin{enumerate}
    \item[(D1)]\label{JC1}  $P$ is a finite left $A$-module, $\chi:P\times P\rightarrow\mathbb{C}^{\times}$ is a function that is additive in each variable
and $\gamma:A^{\varepsilon\text{-sym}}\times P\rightarrow\mathbb{C}^{\times}$ is a function,
\item[(D2)] $\alpha:A^{\times}\to\C^\times$ is a group homomorphism and $f\in\C^\times$.
\end{enumerate}

Furthermore we require that \( \chi \), \( \gamma \), \( \alpha \) and \( f \) satisfy the following:
\begin{equation}\label{chi1}
\alpha(t^{\ast}t) \chi(tx, y) = \chi(x, t^\ast y)\quad x,y\in P, \quad t \in A^{\times},
\end{equation}
\begin{equation}\label{chi2}
 \chi(y,x)=[\chi(x,y)]^{-\va}, \quad x,y\in P,
\end{equation}
\begin{equation}\label{chi3}
\text{ if }\chi(x,y)=1\text{ for all } x\in P \text{ then  } y=0,
\end{equation}
\begin{equation} \label{gamma1}
 \gamma(b+b^{\prime},x)=\gamma(b,x)\gamma(b^{\prime},x),\quad x\in P, \quad b, b'\in A^{\varepsilon\text{-sym}},
 \end{equation}
\begin{equation} \label{gamma2}
\gamma(b,tx)=\gamma(t^{\ast}bt,x),\quad  x\in P, \quad b\in A^{\varepsilon\text{-sym } },\quad  t\in A^{\times},
\end{equation}
\begin{equation}\label{gamma3}
\it\gamma(t,x+z)=\gamma(t,x)\gamma(t,z)\chi(tz,x),\quad x,z\in P, \quad t\in A^{\varepsilon\text{-sym}}\cap A^\times,
\end{equation}
\begin{equation}\label{cez}
f^{2}\left\vert P\right\vert =\alpha(\varepsilon),
\end{equation}
\begin{equation} \label{c}
f\gamma(-\varepsilon t,x)\underset{y\in P}{\sum}\chi(x,y)\gamma(t^{-1},y)=\alpha(-t),\quad x\in P, \quad t\in A^{\varepsilon\text{-sym}}\cap A^\times.
\end{equation}

\begin{lemma}\label{formi} Let $(P,\chi,\gamma)$ be a triple satisfying (D1),  (\ref{chi2}), (\ref{gamma2}) and (\ref{gamma3}). Then
for all $x\in P$ and all $t\in A^{\varepsilon\text{-sym}}\cap A^\times$, we have
$$
\gamma(-\varepsilon t,x)\underset{y\in P}{\sum}\chi( x,y)\gamma(t^{-1},y)=\underset{y\in P}\sum\gamma(t^*,y)=\underset{y\in P}\sum\gamma(t^{-1},y).
$$
\end{lemma}

\begin{proof} Note first of all that (\ref{gamma2}) gives
\begin{equation}
\label{yt}
\gamma(t^{-1},tx)=\gamma(t^*t^{-1}t,x)=\gamma(-\va t,x).
\end{equation}
Therefore by (\ref{gamma3}) and (\ref{yt}), we have
$$
\begin{aligned}
\gamma(-\varepsilon t,x)\underset{y\in P}{\sum}{\chi( x,y)\gamma(t^{-1},y)} &= \underset{y\in P}{\sum}\chi( x,y)\gamma(t^{-1},y)\gamma(t^{-1},tx)\\
&=\underset{y\in P}{\sum}{\chi(x,y)\gamma(t^{-1},y+tx)\chi(x,y)^{-1}}\\
&=\underset{y\in P}{\sum}\gamma(t^{-1},y+tx)=\underset{y\in P}{\sum}\gamma(t^{-1},y)\\
&=\underset{y\in P}{\sum}\gamma(t^{-1},t(x+t^{-1}y))=\underset{y\in P}{\sum}\gamma(-\va t,x+t^{-1}y)\\
&=\underset{y\in P}{\sum}\gamma(t^*,y).
\end{aligned}
$$
\end{proof}

\begin{cor}\label{coro} Let $(P,\chi,\gamma, \alpha,f)$ be a data for $(A,\va)$. Then
$$
\alpha(t^*)=\alpha(t^{-1}),\quad t\in A^{\varepsilon\text{-sym}}\cap A^\times.
$$
Alternatively,
$$
\alpha(t^2)=\alpha(-\va),\quad t\in A^{\varepsilon\text{-sym}}\cap A^\times.
$$
\end{cor}

\begin{proof} Let  $t\in A^{\varepsilon\text{-sym}}\cap A^\times$. Then by (\ref{c}) and  Lemma \ref{formi}
$$
\underset{y\in P}\sum\gamma(t^*,y)=\alpha(-t)=\alpha(\va t^*)/f,
$$
and a fortiori
$$
\underset{y\in P}\sum\gamma(t^{-1},y)=\alpha(\va t^{-1})/f.
$$
Now the conclusion follows from Lemma (\ref{formi}) and the fact that $\alpha$
is a homomorphism.
\end{proof}

\begin{lemma}\label{gae} Let $(P,\gamma)$ be a pair satisfying (\ref{gamma1}) and (\ref{gamma2}). Then
for all $x\in P$ and all $t\in A^{\varepsilon\text{-sym}}$, we have
$$
\gamma(t,-x)=\gamma(t,x)\text{ and }\gamma(t,x)^{-1}=\gamma(-t,x).
$$
\end{lemma}

\begin{proof} The first equation follows from (\ref{gamma2}) by taking $t=-1$, while the second is an immediate consequence of
(\ref{gamma1}).
\end{proof}

\begin{theorem}\label{data} (cf. \cite[Theorem 4.3]{GPS}) Let $(P,\chi,\gamma, \alpha,f)$
    be a data for $(A,\va)$ and let $X$ be complex vector space with basis $(e_v)_{v\in P}$. Let $R$ be the function defined on the Bruhat
elements, with values in $\GL(X)$, as follows:
\begin{equation}\label{luis1}
R(h_{t})e_{x}=\alpha(t)e_{(t^*)^{-1}x},\quad  t\in A^{\times},  x\in P,
\end{equation}

\begin{equation}\label{luis2}
R(u_{b})(e_{x})=\gamma(b,x)e_{x},\quad  b\in A^{\varepsilon\text{-sym }},  x\in P,
\end{equation}

\begin{equation}\label{luis3}
R(\omega)(e_{x})=f\displaystyle{\sum_{y\in P}{\chi(x,y)}e_{y}}, \quad  x,y\in P.
\end{equation}
Then $R$ preserves the relations (R1)-(R6) given in \S\ref{uno}. Thus, if $\SL_*^{\va}(2,A)$ is generated by the Bruhat elements with defining relations (R1)-(R6), then $R$ extends in one and only one way to a representation $\SL_*^{\va}(2,A)\to\GL(X)$.
\end{theorem}

\begin{proof} A direct application of (D2) (resp. (\ref{gamma1})) shows that $R$ preserves (R1) (resp. (R2)). As for (R3), the definition (\ref{luis3}) yields
\begin{equation}\label{luis25}
R(\omega)R(\omega)e_x=f^2\underset{z\in P}\sum
(\underset{y\in P}\sum   \chi(x-\va z,y))e_z.
\end{equation}
By  (D1), (\ref{chi2}) and (\ref{chi3}), the map $y\mapsto \chi(x',y)$ is a nontrivial  linear character of $P$ whenever $x'\neq 0$. Thus
$$
\underset{y\in P}\sum   \chi(x-\va z,y)=\begin{cases} 0 &\text{ if } z\neq \va x,\\\left\vert P\right\vert & \text{ if } z=\va x. \end{cases}
$$
Therefore (\ref{cez}) and (\ref{luis25}) give
$$
R(\omega)R(\omega)e_x=f^2\left\vert P\right\vert e_{\va x}=\alpha(\va)e_{\va x}=R(h_{\va})e_x.
$$
Regarding (R4), let $t\in A^{\times}$ and $r\in A^{\varepsilon\text{-sym}}$.  Then by (\ref{gamma2}) we have
$$
R(u_{trt^*})R(h_t)e_x=\alpha(t)\gamma(trt^*,(t^*)^{-1}x)e_{(t^*)^{-1}x}=\alpha(t)\gamma(r,x)e_{(t^*)^{-1}x}=R(h_t)R(u_r)e_x.
$$
In regards to (R5), making use of (D2) and (\ref{chi1}) we see that
$$
R(\omega)R(h_t) e_x=f\alpha(t)\displaystyle{\sum_{y\in P}\chi((t^*)^{-1}x,y)e_{y}}=f\alpha(t^*)^{-1}
\displaystyle{\sum_{y\in P}\chi(x,t^{-1}y)e_{y}}.
$$
On the other hand, we have
$$
R(h_{(t^*)^{-1}})R(\omega)e_x=f\alpha(t^*)^{-1}\displaystyle{\sum_{y\in P}\chi(x,y)e_{ty}}=f\alpha(t^*)^{-1}\displaystyle{\sum_{y\in P}\chi(x,t^{-1}y)e_{y}}.
$$
These two expressions are identical.

Next let  $t\in A^{\varepsilon\text{-sym}}\cap A^\times$. We wish to verify that $R$ preserves (R6), or the equivalent relation obtained by replacing $t$ by $-\va t$, namely
$$u_{-\va t} \omega u_{t^{-1}}\omega = \omega h_{\va t^{-1}}u_{\va t}.$$
Let $x\in P$. Applying (\ref{luis2}) and (\ref{luis3}) and making use of (\ref{chi2}) we obtain
\begin{equation}\label{lado}
R(u_{-\va t})R(\omega)R(u_{t^{-1}})R(\omega)e_x=f^2\underset{z\in P}\sum \gamma(-\va t,z)\big(\underset{y\in P}
\sum \chi(x-\va z,y)\gamma(t^{-1},y)\big)e_z.
\end{equation}
In view of (D1), (\ref{gamma3}) and Lemma \ref{gae}, we have
$$
\gamma(-\va t,x-\va z)=\gamma(-\va t,x)\gamma(-\va t,z)\chi(tx,z)
$$
Appealing to (D2) and Lemma \ref{gae}, we can translate this as follows:
\begin{equation}\label{lado1}
\gamma(-\va t,z)=\gamma(\va t,x)\chi(-tx,z)\gamma(-\va t,x-\va z).
\end{equation}
Substituting (\ref{lado1}) in (\ref{lado}) and making use of (\ref{c}) and Lemma \ref{gae} we derive
\begin{equation}\label{lado2}
R(u_{-\va t})R(\omega)R(u_{t^{-1}})R(\omega)e_x=f\alpha(-t)\gamma(\va t,x)\underset{z\in P}\sum \chi(-tx,z)e_z.
\end{equation}

On the other hand, applying (\ref{luis1})-(\ref{luis3}), we see that
\begin{equation}\label{lado3}
R(\omega)R(h_{\va t^{-1}})R(u_{\va t})e_x=f\alpha(\va t^{-1})\gamma(\va t,x)\underset{z\in P}\sum  \chi(-tx,z)e_z.
\end{equation}
By Corollary \ref{coro}, $\alpha(-t)=\alpha(\va t^{-1})$. Therefore, (\ref{lado2}) and (\ref{lado3})
are identical.
\end{proof}
%\bigskip

\begin{theorem}\label{main} Let $B$ be a finite ring where $2\in B^\times$ having an involution $*$ and a primitive linear character $\beta:R^+\to\C^\times$
satisfying $\beta(b+\va b^*)=1$ for all $b\in B$, and extend $*$ to an involution, also denoted by $*$, of $A=M(m,B)$. Moreover, suppose $m$ is even if $\va=1$. Let $P$ be the column space $B^m$ and
let $X$ be complex vector space with basis $(e_v)_{v\in P}$. Set
\begin{itemize}
\item  $\chi(a,b)=\beta(2a^*b)$,  for all $a,b\in P$,
\item $\gamma(S,a)=\beta(-\va a^*Sa)$, for all $a\in P$, and all $S\in A$ satisfying $S^*+\va S=0$,
\item $\alpha=\mu$,
\item $f=\widehat{G}(\beta)^{-m}(-1)^{\frac{|N|-1}{2}}$ if $\va=-1$, while $f=1/|B|^n$ if $\va=1$.
\end{itemize}
Then  $(P,\chi,\gamma, \alpha,f)$  is a data for $(A,\va)$. Moreover, associated to this data, there exists one and only one representation
$W:\SSL_*^{\va}(2, A)\to\GL(X)$ satisfying (\ref{luis1})-(\ref{luis3}),
namely the one satisfying (\ref{fer1})-(\ref{fer3}) in the skew hermitian case and (\ref{ferna1})-(\ref{ferna3}) in the hermitian case. In other words, the Weil representations
of $\SSL_*^{\va}(2, A)$ obtained via abstract data and through Heisenberg groups are identical.

If $B$ is local we can replace the two instances of $\SSL_*^{\va}(2, A)$ above by $\SL_*^{\va}(2, A)$, except only when $\va=1$ and $*$ is ramified, in which case
$[\SL_*^{+}(2, A):\SSL_*^{+}(2, A)]=2$.
\end{theorem}

\begin{proof} It is clear that $\chi$ is bi-additive and we already established the fact that $\alpha$ is a group homomorphism. Moreover, by definition, we have
$$
\chi(tx, y)=\beta(2x^* t^*y)=\chi(x, t^\ast y)\quad x,y\in P, t \in A.
$$
But $\alpha(t)^2=1$ and, by (\ref{vamu}), $\alpha(t)=\alpha(t^*)$, for all $t\in A^\times$, so (\ref{chi1}) holds. The definition of $\chi$
and (A6) immediately yield (\ref{chi2}). Since $2\in B^\times$, $\beta$ is primitive and $h$ is nondegenerate, we see that (\ref{chi3}) is satisfied. A trivial calculation shows that (\ref{gamma1}) and (\ref{gamma2}) hold. Now, given any
$x,z\in P$ and $t\in A^{\varepsilon\text{-sym}}\cap A^\times$, we have
$$
\gamma(t,x+z)=\beta((-\va)(x+z)^*t(x+z))=\gamma(t,x)\gamma(t,z)\beta((-\va)(x^* t z+z^*tx)).
$$
But, thanks to $t^*=-\va t$ and (A6), we have
$$
\beta(x^* t z+z^*tx)=\beta(x^* t z-\va z^*t^*x)=\beta(x^* t z-\va (x^*tz)^*)= \beta(2x^* tz)=\chi(x,tz).
$$
Applying (\ref{chi2}) we obtain (\ref{gamma3}).

In the hermitian case, the very definition of $f$ shows that $f^2|P|=1=\alpha(1)$. In the skew hermitian case, (\ref{mu9}) gives
$$
f^2|P|=\alpha(-1).
$$
This establishes (\ref{cez}).

Finally, let $x\in P$ and $t\in A^{\varepsilon\text{-sym}}\cap A^\times$. By Lemma \ref{formi}, we have
$$
f\gamma(-\varepsilon t,x)\underset{y\in P}{\sum}\chi( x,y)\gamma(t^{-1},y)=f\underset{y\in P}\sum\gamma(t^*,y).
$$
If $\va=-1$ then (\ref{mu82}) and our definition of $f$ yield
$$
f\underset{y\in P}\sum\gamma(t^*,y)=\alpha(-1)\alpha(t^*)=\alpha(-t),
$$
while if $\va=1$ then (\ref{mu89}) and our definition of $f$ give
$$
f\underset{y\in P}\sum\gamma(t^*,y)=\alpha(t^*)=\alpha(-t).
$$
This proves (\ref{c}).

Comparing (\ref{fer1})-(\ref{fer3}) and (\ref{ferna1})-(\ref{ferna3}) with (\ref{luis1})-(\ref{luis3})
we see that $W$ and $R$ agree on the Bruhat elements. But $W$ {\em is} a representation and the Bruhat elements
generate $\SSL_*^{\va}(2, A)$, so $R$ extends
in exactly one way to a homomorphism $\SSL_*^{\va}(2, A)\to\GL(X)$, namely the one defined by (\ref{fer1})-(\ref{fer3}) and (\ref{ferna1})-(\ref{ferna3}).

The last assertion of the theorem is established in Theorems \ref{jopa4}.
\end{proof}

%\begin{note}  {\color{red} Suppose $B$ is local, $\va=1$ and $m=2n$ is even. We claim that the that index of $\SSL_*^{+}(2, A)$ in \( \SL_*^{+}(2,A) \) is equal to 2.
%    To be continued.}
%\end{note}

%I suspect that when $\va=1$ the data axioms are not verified and $W$ is not equivalent to $R$. This is because of (\ref{pu2}) or (\ref{tu4}),
%which are incompatible with (\ref{luis3}) when $\va=1$. Which of the data axioms is not verified? Most likely (\ref{c}) or
%its equivalent version coming from Lemma \ref{formi}. But according to \S\ref{sech} these are satisfied, so I think there is something wrong
%in that section.

\begin{note} The first part of Theorem \ref{main} shows that our choice of $(P,\chi,\gamma, \alpha,f)$
satisfies all data axioms. Therefore, Theorem \ref{data} gives an independent verification that $W$ does preserve all relations (R1)-(R6).
\end{note}

%{\color{red}
%The examples below shows that $\SL_*^{+}(2, A)$ need not be generated by its Bruhat elements even if $B$ is local. Note that if $A=B$ (i.e. $m=1$) is local then $\SL_*^{+}(2, A)$ is indeed
%generated by Bruhat elements. Is it true that index of $\SSL_*^{+}(2, A)$ is at most 2 when $B$ is local?
%}

\section{Examples}\label{exam}

\begin{example} Let $\O$ be discrete valuation ring with involution having a finite residue field of characteristic not 2 and let~$B$ be a quotient of $\O$ by a nonzero power of its maximal ideal. Then $B$ inherits an involution, say $*$, from $\O$
and we let $R$ stand for the fixed ring of $*$. Three cases arise (see \cite[Proposition 5]{CQS}):

$\bullet$ symplectic: $*$ is trivial, that is, $B=R$.

$\bullet$ unramified: $B=R\oplus \theta R$, where $\theta$ is a unit of $B$ and $\theta^*=-\theta$.

$\bullet$ ramified: $B=R\oplus \pi R$, where $B\pi$ is the maximal ideal of $B$ and $\pi^*=-\pi$.

%The nilpotency degree of the maximal ideal of $A$ (and $R$) in the symplectic and unramified cases will be denoted by $\ell$. Context will allow us to avoid confusion.

The ramified case further divides into two cases, odd or even, depending on whether the nilpotency degree of $\pi$ is odd, $2\ell-1$, or even,~$2\ell$.

In all cases, $B$ and $R$ are finite, commutative, principal, local rings of
odd characteristic. Let $\r=B\pi$ and $\m=Rp$ stand for their
maximal ideals, so that $\m=R\cap\r$, and let $F_q=R/\m$ be the
residue field of $R$. Then $B/\r\cong F_q$ in the symplectic and
ramified cases, and $B/\r\cong F_{q^2}$ in the unramified case. We
choose $\pi$ and $p$ so that $\pi=p$ in the symplectic and
unramified cases, and $\pi^2=p$ in the ramified case. %The minimal ideal of $B$ will be denoted by $\min$.

We have $B=R\oplus B_s$, where $B_s$ is the additive group of all skew hermitian elements of~$B$. In the unramified case, $B_s=R\theta$
and $\{1,\theta\}$ is an $R$-basis of $B$. In the ramified case, $B_s=R\pi$, but $\{1,\pi\}$ is an $R$-basis of $B$ in the even case only.
In the ramified odd case the annihilator of $\pi$ in $R$ is $Rp^{\ell-1}$. This is true even in the extreme case when $\ell=1$,
which is the symplectic field case $A=R=F_q$.

Let $d:B\to R$ be the projection of $B$ onto $R$ in the symplectic, unramified, and ramified odd cases, and $d(r+s\pi)=s$ in the ramified even case.

We take $h$ be skew hermitian ($\va=-1$) in the symplectic, unramified, and ramified odd cases,
whereas $h$ is hermitian ($\va=1$) in the ramified even case.

It is easy to see that $R$ admits a primitive group homomorphism $\lambda:R^+\to\C^\times$, in which case so
is $\beta=\lambda\circ d:B^+\to \C^\times$.

If we set $f=d\circ h:V\times V\to R$ we obtain an embedding of the unitary group associated to $h$
into the symplectic group associated to the nondegenerate alternating form $f$.

See \cite{GV} for a comparison between G\'erardin's method \cite{G} and the abstract data construction of the Weil representation
of $\U_{2n}(F_{q^2})$, a special instance of the unramified case.
\end{example}

\begin{example} Let $F_q$ be a finite field of odd characteristic and let $F_{q^2}$ be its quadratic extension.
We have an involution $*$ of $F_{q^2}$ with fixed field $F_q$, given by $a^*=a^q$. Let $C$ be the skew polynomial
ring over $F_{q^2}$, where $ta=a^* t$ for all $a\in F_{q^2}$. It is well known and easy to see that $C$ is a left and right principal
ideal domain. There is a unique extension of $*$ to an involution of $C$ such that $t^*=-t$,
given by $(a_0+a_1 t+a_2t^2+a_3t^3+\cdots)^*=a_0^*-a_1t+a_2^*t^2-a_3t^3+\cdots$. For $s\geq 2$, set $B=C/(t^s)$. The local ring $B$
inherits an involution, also denoted by $*$, from $C$. Note that $B$ has a unique minimal right (and left) ideal, namely $(t^{s-1})$
(we abuse notation here). Let $\beta:B\to\C^\times$ be the group homomorphism defined by
$$
\beta(a_0+a_t+\cdots+a_{s-1} t^{s-1})=\lambda(a_{s-1}+a_{s-1}^*),\quad a_i\in F_{q^2},
$$
where $\lambda:F_q^+\to\C^\times$ is a primitive (=nontrivial) group homomorphism. Set $\va=-1$ if $s$ is odd and $\va=1$ if $s$ is even.
Then all our axioms (A1)-(A6) are satisfied. This gives an example of a noncommutative local ring $B$ satisfying all our axioms, and both
the hermitian and skew hermitian cases occur.
\end{example}

%{\color{red}
%Include one more example of noncommutative rings from [HS] (none are local).
%}

\begin{example}\label{notlocal} Let \( F_q \) be a finite field of odd characteristic and consider the non local ring \( B = M(2,F_q) \). Set \( m=1 \)
so that \( A = B \). For
\( x \in B \) let \( x^* \) be the adjugate of \( x \), i.e.
\[ \begin{pmatrix} x_{11} & x_{12} \\ x_{21} & x_{22} \end{pmatrix}^*
= \begin{pmatrix} x_{22} & -x_{12} \\ -x_{21} & x_{11} \end{pmatrix}.\]
Equivalently, \( * \) is the adjoint with respect to the standard
symplectic form on \( F_q^2 \).
Let \( \beta=\lambda\circ\mathrm{tr} \), where $\lambda:F_q^+\to\C^*$ is a primitive (=nontrivial) group
homomorphism and $\mathrm{tr}:M(2,F_q)\to F_q$ is the trace map. Let \( \varepsilon = -1 \).
All our axioms (A1)-(A6) are satisfied with these choices. Now let
$$ a = \begin{pmatrix} 0 & 0 \\ 0 & 1 \end{pmatrix},\;
 b = \begin{pmatrix} 0 & 1 \\ 0 & 0 \end{pmatrix},\;
c = \begin{pmatrix} 0 & 0 \\ 1 & 0 \end{pmatrix},\;
d = \begin{pmatrix} 1 & 0 \\ 0 & 0 \end{pmatrix}.
$$
Using (\ref{eqelements}), we see that \( X = \begin{pmatrix}
    a & b \\ c & d
\end{pmatrix} \in  \SL_*^{-}(2,A)\). However  $X \not\in \SSL_*^{-}(2,A)$, since all elements of \( \SSL_*^{-}(2,A) \) have determinant
\( 1 \)
(when considered as elements of \( M(4,F_q) \)), whereas \( X \) has
determinant \( -1 \).

It is also easy to see that the coprime
lemma \cite[Proposition 3.3]{PS}
does not hold for \( a,c \) as above, so indeed $\SL_*^{-}(2,A)$
cannot generated by the Bruhat elements due to \cite[Lemma 3.5]{PS}. As a matter of fact, in this example
\( A^s = \{ \lambda I_2 : \lambda \in F_q\} \) and it is clear that
\( a + rc \not\in A^\times \) for any \( r \in A^s \).
\end{example}

%==================================================

\end{document}